\documentclass[12pt]{amsart}
\usepackage{amsmath}
\usepackage{amssymb}
\usepackage{amsthm}
\usepackage{mathrsfs}
\usepackage{comment}
\usepackage{hyperref}

\usepackage[all,cmtip]{xy}\usepackage{xcolor}
\usepackage{enumerate}
\usepackage{bm}
\usepackage{dsfont}
\usepackage{mathtools}
\usepackage[cal=euler]{mathalfa}
\usepackage[top=1in, bottom=1.25in, left=1.25in, right=1.25in]{geometry}
\usepackage{parskip}

\hypersetup{colorlinks=true,linkcolor=magenta,citecolor=blue}

\DeclareMathAlphabet{\mathpzc}{OT1}{pzc}{m}{it}

\theoremstyle{plain}
\newtheorem{theorem}{Theorem}[section]
\newtheorem*{theorem*}{Theorem}

\newtheorem{lemma}[theorem]{Lemma}
\newtheorem*{claim*}{Claim}
\newtheorem{proposition}[theorem]{Proposition}

\newtheorem{corollary}[theorem]{Corollary}

\theoremstyle{definition}

\newtheorem{remark}[theorem]{Remark}

\numberwithin{equation}{section}
\numberwithin{figure}{section}

\newcommand{\Cent}{\mathrm{Cent}}

\newcommand{\Gal}{\mathrm{Gal}}

\newcommand{\Aut}{\mathrm{Aut}}

\newcommand{\ignore}[1]{}

\begin{document}
\title{The isotopy classes of Petit division algebras}

\author{Susanne Pumpl\"un}
\address{School of Mathematical Sciences, University of Nottingham,
Nottingham NG7 2RD}
\email{Susanne.Pumpluen@nottingham.ac.uk}



\date{\today}
\maketitle

\begin{abstract}
Let $R=K[t;\sigma]$ be a skew polynomial ring, where $K$ is  a cyclic Galois field extension of degree $n$ with Galois group generated by $\sigma$.   We show that two irreducible skew polynomials $f,g\in R$ are similar if and only if they have the same bound.
We prove that for two irreducible similar skew polynomials $f,g\in R$ the nonassociative Petit division algebras  $R/Rf$ and $R/Rg$ are isotopic. We then refine this result and demonstrate that $f$ and $g$
 also yield two isotopic nonassociative Petit algebras  $R/Rf$ and $R/Rg$, when the two irreducible polynomials in $F[x]$ that define the minimal central left multiples of $f$ and $g$
 have identical degree and lie in the same orbit of some group $G$.  For finite fields, we obtain the exact number of isotopy classes of Petit semifields for $n>2$.
  \end{abstract}

\section{Introduction}

Families of nonassociative rings and algebras can be classified up to isomorphism or, much more generally, up to isotopy. Both times, the structure of the algebra is preserved to a large extent. Two algebras which are isomorphic also are isotopic, the converse holds only in very special cases.  Semifields (i.e., division algebras over finite fields) are usually investigated and classified only up to isotopy, due to their intimate connection with projective planes that can be classified only knowing the algebra's isotopy class.

Moreover, there is a one-one correspondence between division algebras and square linear maximum-rank distance (MRD) matrix codes  \cite{CKWW2015} and isotopic algebras yield equivalent codes.
Equivalent matrix codes have the same behaviour and invariants, hence in order not to duplicate known matrix codes it is important to know their equivalence classes.

Let $K/F$ be a cyclic Galois field extension of degree $n$ with Galois group generated by $\sigma$, and consider the skew polynomial ring $R=K[t;\sigma]$. We will look at the question when two Petit division algebras $R/Rf$ and $R/Rg$ are isotopic. Curiously, the methods we present here  do not extend to the case of arbitrary Petit algebras. We need $f$ to be irreducible, thus will only cover results on the isotopy of Petit division algebras.

 Petit division algebras $R/Rf$ \cite{P66} are defined employing irreducible skew polynomials $f\in R$, and make up one of the largest known classes of semifields \cite{LS, JJ}. Subsequently, they were  investigated  over general base fields and rings, e.g.  \cite{P.codes, P16}. The structure of Petit algebras strongly reflects the properties of the skew polynomial $f$, for instance the right nucleus of $R/Rf$ is the eigenspace of $f$.

 For one small family of Petit division algebras over finite fields isotopy and isomorphism are equivalent notions \cite[Theorem 9]{Sandler1962}. This result was recently extended to arbitrary base fields \cite{NevinsPumpluen2025San}. We now also know that the choice of the generator of ${\Gal}(K/F)$ is sometimes relevant in the algebra construction $K[t;\sigma]/K[t;\sigma]f$ as well.
 More precisely,  we showed that  distinct  generators $\sigma_1$  and $\sigma_2$  of ${\rm Gal}(K/F)$  produce  non-isomorphic families of Petit rings $K[t;\sigma_1]/K[t;\sigma_1](t^m-a_1)$ and $K[t;\sigma_2]/K[t;\sigma_2](t^m-a_2)$, when ${\Aut}(K)$ is abelian and $n \nmid m$ or $n=m$ and $a_i\in K\setminus F$
 \cite[Theorem 3.1, Theorem 3.2]{NevinsPumpluen2025San}.

In this paper, we first prove that two irreducible skew polynomials $f$ and $g$ in $R$ are similar if and only if their bounds $f^*$ and $g^*$ are the same
(Theorem \ref{thm:irred}).

 If $f$ and $g$ are irreducible and similar then the nonassociative division rings  $R/Rf$ and $R/Rg$ are isotopic (Theorem \ref{thm:easy}).

 Let $I(F,m_0)$ denote the set of all monic irreducible polynomials in $F[x]$ of degree $m_0$.
We then take the semidirect product
$$G=N_{K/F}(K^\times)\rtimes \{\tau\in {\rm Aut}(K)\,|\, \sigma\circ\tau=\tau\circ \sigma \text{ and } \tau|_F\in {\rm Aut}(F) \}$$
$$=\{(\lambda,\tau)\,|\,\lambda\in N_{K/F}(K^\times),\tau\in {\rm Aut}(K)\text{ such that }\sigma\circ\tau=\tau\circ \sigma \text{ and } \tau|_F\in {\rm Aut}(F) \},$$
 and define an action of $G$ on $I(F,m_0)$ via
$$g^{(\lambda,\tau)}(x)=\lambda^{-m_0}g^\tau(\lambda x),$$
 where  $g^\tau(x)=\sum_{i=0}^s \tau(c_i) x^i$ for $g(x)=\sum_{i=0}^s c_ix^i \in F[x]$. Our Main Theorem \ref{thm:main12}  concludes that for any two  irreducible $f,g\in R$ of degree $m$ with   monic bounds
 $f^*=\hat{h}_1(t^n)$ and $g^*=\hat{h}_2(t^n)$ and $\hat{h}_1,\hat{h}_2\in F[x]$ both monic of the same degree $m_0$,
  such that $\hat{h}_1$ and $\hat{h}_2$ lie in the same $G$-orbit of $I(F,m_0)$, the Petit algebras $R/Rf$ and $R/Rg$  are isotopic division rings.

   This generalises \cite[Theorem 12]{LS} which was proved for finite fields. In this general setting, however, the degree $m_0$ of the irreducible monic polynomial  $\hat{h}(x)\in F[x]$ does not automatically predetermine the degree of $f$ to be $m=m_0$ as in the finite field case. In fact, different irreducible monic polynomials  $\hat{h}(x)\in I(F,m_0)$ will define bounds $h(t)=\hat{h}(t^n)$ of skew polynomials $f\in R$ with corresponding Petit algebras $R/Rf$, where the degree of $f$ (and hence the dimension of the algebra) will depend on the number $k$ of irreducible factors of $h(t)=\hat{h}(t^n)$ in $R$, and $f$ need not be irreducible just because $h(x)$ is.

    When ${\rm Aut}(K)$ is abelian,  and $f(t)=t^m-a$ is irreducible with bound $f^*=\hat{h}(t^n)$ then
    for all $g(t)=t^m-b$ with bound $g^*$ defined via some $\hat{h_0}(x)$ that lies in the $G$-orbit of $\hat{h}(x)$,  we have $K[t;\sigma]/K[t;\sigma](t^m-a)\cong K[t;\sigma]/K[t;\sigma](t^m-b)$, whenever
  either  $n > m\geq 3$ and  $\sigma(a)\neq a$ and $\sigma(b)\neq b$,
  or
  $n=m$ and $a$ and $b$  do not lie in any proper subfield of $K$ (Theorem \ref{thm:maincor}).

 We thus showed that when two irreducible polynomials $\hat{h}_1,\hat{h}_2\in F[x]$  of identical degree  lie in the same  orbit of $G$, then for the
 irreducible skew polynomials in $f,g\in R=K[t;\sigma]$ whose bounds are given by
 $f^*=\hat{h}_1(t^n)$ and $g^*=\hat{h}_2(t^n)$, the associated MRD codes $C=\{R_a\,|\,  a\in R/Rf\}\subset M_k(K)$ and $C'=\{R_a\,|\,  a\in R/Rg\}\subset M_k(K)$ are equivalent (Theorem \ref{t:main2}). These matrix codes are examples of the codes proposed in
\cite{Sheekey2019}.

We finish by explicitly computing the number of distinct orbits on the set $I(q,m)$ under the action of $G$ for finite fields (Theorem \ref{thm:countingorbits}).

The isotopies and autotopism classes of  Petit semifields were determined by Santonastaso and Zhou \cite{SantonastasoZhou2026}. When  $n>2$ they give a lower bound for the number of isotopy classes.

We obtain the exact number of isotopy classes for $n>2$.

\begin{theorem*} (Theorem \ref{thm:mainmain})
 Let $n>2$. Then the number of isotopy classes of semifields $R/Rf$ of order $q^{nm}$ obtained from irreducible skew polynomials $f$ of degree $m$ in $ \mathbb{F}_{q^n}[t;\sigma]$ is exactly
 The number  of distinct $G$-orbits on $I(q,m)$ is
$$M(q,m)=\frac{1}{(q-1)hm}\sum_{\lambda\in \mathbb{F}_q^\times}\sum_{r=0}^{h-1} \sum_{j=0}^{m-1}
 \sum_{d|m}\mu(d) N_{(r,j)}(\frac{m}{d},\lambda).$$
 \end{theorem*}

\section{Preliminaries}

 \subsection{Nonassociative algebras}

Let $F$ be a field. In the following we will consider  nonassociative rings and algebras over $F$. A nonassociative ring or $F$-algebra $A$ is called a \emph{proper} nonassociative ring, respectively algebra, if it is not associative.

A nonassociative ring $A\not=0$ is called a \emph{division ring} if for any $a\in A$, $a\not=0$, the left multiplication  with $a$, $L_a(x)=ax$, and the right multiplication with $a$, $R_a(x)=xa$, are bijective.  A \emph{semifield} is a unital nonassociative division ring with finitely many elements.
 An $F$-algebra $A\not=0$ is called a \emph{division algebra} if $A$ is a division ring.
 If $A$ has finite dimension over $F$, $A$ is a division algebra if
and only if $A$ has no zero divisors \cite[pp. 15, 16]{Sch}.

Let $A$, $A'$ be two unital nonassociative division rings.
  The rings $A, A'$ are called {\it isotopic}, if there exist bijective additive  maps $f,g,h:A\rightarrow A'$, each of them linear over some
subfield of the center of  $A$, such that
 $h(xy)=f(x) g(y)$ for all $ x,y\in A.$ If $f=g=h$ then $A\cong A'$.

 Two algebras $A, A'\in {\rm Alg}(V)$ are called {\it isotopic}, if
$x\circ_A y=h(f(x) \circ_{A'} g(y))$ for all $ x,y\in V.$ If $f=g=h^{-1}$ then $A\cong A'$.

\subsection{Twisted polynomial rings}

Let $K/F$ be a cyclic Galois field extension of degree $n$ with Galois group generated by $\sigma$.
 The \emph{twisted polynomial ring} $R=K[t;\sigma]$
is the ring of polynomials $\{a_0+a_1t+\dots +a_nt^n\,|\, a_i\in K\}$, with term-wise addition and multiplication given by the rule
$ta=\sigma(a)t$ for all $a\in K$ \cite{O}; see also \cite[Chapter I]{J96}. The constant nonzero polynomials $K^\times$ are the units of $R$.

For $f=a_0+a_1t+\dots +a_nt^n\in R$ with $a_n\not=0$ define ${\rm
deg}(f)=n$ and put ${\rm deg}(0)=-\infty$. Then ${\rm deg}(fg)={\rm deg}
(f)+{\rm deg}(g).$
 An element $f\in R$ is called \emph{irreducible} if  it is not a unit and it has no proper factors, \emph{i.e.} there do not exist
 $g,h\in R$, neither a unit, such that $f=gh$.

 There exists a right division algorithm in $R=K[t;\sigma]$: for all $g,f\in R$, $f\neq 0$, there exist unique $r,q\in R$
  such that ${\rm deg}(r)<{\rm deg}(f)$ and $g=qf+r$ \cite[\S1.1]{J96}.
(Our terminology is the one used by Petit \cite{P66} and
 different from Jacobson's \cite{J96}, who calls what we call right a left division algorithm and vice versa.)
 For all $g \in R$ let $g\, {\rm mod}_r f$ denote the remainder of $g$ upon right division by $f$.

%%%%%%%%%%%%%%%%%%%%%%%%%%%%%%%%%%%%%%%%%%%%%%%%%%%%%%%%%%%%%%%%%%%%%%%%%

\subsection{Nonassociative algebras obtained from twisted polynomial rings} \label{subsec:structure}

%%%%%%%%%%%%%%%%%%%%%%%%%%%%%%%%%%%%%%%%%%%%%%%%%%%%%%%%%%%%%%%%%%%%%%%%%%%

In this paper, let $K/F$ be a cyclic Galois field extension of degree $n$ with Galois group ${\rm Gal}(K/F)=\langle \sigma \rangle$ and let $R=K[t;\sigma]$. Let $f \in R$ have degree $m$.

 The set $\{g\in R\,|\, {\rm deg}(g)<m\}$ endowed with the usual term-wise addition of polynomials and the multiplication $g \circ h = gh \, {\rm mod}_r f$, where the right-hand side is the remainder of dividing the polynomial $gh$ by $f$ from the right, is a unital nonassociative ring we denote by $S_f$ or $R/Rf$. We usually will simply use juxtaposition for writing the multiplication in $R/Rf$. Also, $R/Rf$ is a unital nonassociative algebra over $F$
\cite{P66}.
 We call $R/Rf$ a {\it Petit algebra}.
 For all $a \in K^{\times}$ we have $R/Rf = R/R(af)$, and if $f\in R$ has degree 1 then $R/Rf \cong K$.
In the following, we thus assume that $f$ is monic and that it has degree $m\geq 2$, unless specifically mentioned otherwise. The algebra
 $R/Rf$ is associative if and only if $f$ is \emph{right invariant}, i.e. if $Rf$ is a two-sided ideal in $R$.
In that case, $R/Rf$ is equal to the classical associative quotient algebra $R/(f)$. Note that  $f(t) = t^m - \sum\limits_{i=0}^{m-1}a_it^i \in R$ is right invariant if and only if $a_i \in F$ and $a_i(\sigma^m(d)-\sigma^i(d))=0$ for all $i \in \lbrace 0,1,\dots,m-1 \rbrace$ and for all $d \in K$ \cite[(15)]{P66}.

By  \cite[Theorem 1.1.22]{J96}, this means that $f$ is right invariant in $R$ if and only if $f(t)=ag(t)t^s$ for some $g \in C(R)$, $a\in K$, and some
integer $s\geq 0$.

If $R/Rf$ is not associative then ${\rm Nuc}_l(R/Rf)={\rm Nuc}_m(R/Rf)=K$
and $C(S_f)=F$.   Moreover, then
$${\rm Nuc}_r(R/Rf)=\{g\in R\,|\, {\rm deg}(g)<m \text{ and }fg\in Rf\}.$$
 is the eigenspace of $f\in R$ \cite{P66}.

The Petit algebra $R/Rf$ is a division algebra, if and only if $f$ is irreducible in $R$, if and only if ${\rm Nuc}_r(R/Rf)$ is a division algebra.  It is well known that each nontrivial zero divisor $q$ of $f$ in ${\rm Nuc}_r(R/Rf)$ gives a proper factor ${\rm gcrd}(q,f)$ of $f$, where ${\rm gcrd}(q,f)$ denotes the greatest common right divisor of $q$ and $f$ in $R$,  e.g. see \cite{GLN13}.

\begin{proposition}\label{thm:7}
Let $F_0\subset F$ be any subfield such that $\Aut_{F_0}(K)$ is abelian.
Let  $\tau\in \Aut_{F_0}(K)$.
\\ (i) The map $G:R\rightarrow R$ defined via $G(t)=\sum_{i=0}^{l} k_i t^i$, $k_i \in K$, extends $\tau$ to an $F_0$-algebra isomorphism $G: K[t;\sigma] \rightarrow K[t;\sigma]$ if and only if
$G(t)=\alpha t$ for some $\alpha \in K^\times$ and $\sigma\circ \tau =\tau\circ \sigma$. In that case,
$$G(\sum_{i=0}^l b_it^i)=\sum_{i=0}^l \tau(b_i) (\alpha t)^i.$$
We write $G=G_{\tau,\alpha}$.
\\ (ii) Let $G_{\tau,\alpha}:R\rightarrow R$ be an $F_0$-algebra isomorphism extending $\tau$, then $R/Rf\cong R/R G_{\tau,\alpha}(f)$ as nonassociative $F_0$-algebras.
\end{proposition}

When $\tau\in \Aut_{F_0}(K)$  and $g(t)=\sum_{i=0}^l b_it^i\in R$ we will use the notation $g^{\tau}$ for $\sum_{i=0}^l \tau(b_i) t^i$.

\begin{proof}
$(i)$
By \cite[Theorem 3]{Ri}, the map $G:R\rightarrow R$ defined via $G(t)=\sum_{i=0}^{l} k t^l$, $k_i \in K$, extends $\tau\in \Aut_{F_0}(K)$ to an isomorphism $G: K[t;\sigma] \rightarrow K[t;\sigma]$ if and only if
$G(t)=\alpha t$ for some $\alpha \in K^\times$ and $\sigma\circ \tau =\tau\circ \sigma$.
\\ $(ii)$ Define $R_m=\{g\in R\, |\,\deg(g)<m\}$.
Every isomorphism $G:R\rightarrow R$ extending $\tau\in \Aut_{F_0}(K)$  canonically induces an isomorphism between the nonassociative $F_0$-algebras $K[t;\sigma]/K[t;\sigma]f$  and  $ K[t;\sigma]/K[t;\sigma]G(f)$, since $G|_{R_m}:R_m\rightarrow R_m$ and   $G(g\circ_f h)=G(gh -s f)=G(g)G(h)-G(s)G(f)=G(g)\circ_{G(f)} G(h)$ for some $s\in R_m$.
\end{proof}

 When $n \geq m-1$ and  $R/Rf$ and $R/Rg$ are two  proper  Petit algebras, then every $F_0$-algebra isomorphism $G:R/Rf\rightarrow R/Rg$  is induced by an isomorphism $G_{\tau,\alpha}: K[t;\sigma]\rightarrow K[t;\sigma]$ which restricts to some $\tau\in\Aut_{F_0}(K)$ that commutes with $\sigma$ \cite{Pumpluen2025}.

\section{Isotopic Petit algebras}

A twisted polynomial  $f \in R$ is \emph{bounded} if there exists a nonzero polynomial $f^* \in R$,  such that $Rf^*$ is the largest two-sided ideal of $R$ contained in $Rf$. The polynomial $f^*$ is uniquely determined by $f$ up to scalar multiplication by an element in $K^\times$ and is called the \emph{bound} of $f$.

 Since we assume that $\sigma$ has finite order $n>1$, every $f\in R$ is bounded and $R$ has center $C(R) = F[t^n]\cong F[x]$, where $x=t^n$ \cite[Theorem 1.1.22]{J96}.

For any
$f \in R$ we  define the \emph{minimal central left multiple} ${\rm mclm}(f)$ of $f$ in $R$ as the unique  polynomial $h(t)$ of minimal degree in $F[t^n]$, such that $h = gf$ for some $g \in R$, and such that $h(t)=\hat{h}(t^n)$ for some monic $\hat{h} \in F[x]$. If the greatest common right divisor ${\rm gcrd}(f,t)$ of $f$ and $t$ is one, then $f^*\in C(R)$ \cite[Lemma 2.11]{GLN13}, and the minimal central left multiple $h(t)=\hat{h}(t^n)$ of $f$ equals $f^*$ (this bound $f^*$ is unique up to a scalar multiple from $K^\times$).
The assumption that ${\rm gcrd}(f,t)=1$ is trivially satisfied when $f$ is irreducible of degree $m\geq 2$.

We denote the minimal central left multiple  of $f$ by $h(t)=\hat{h}(t^n)$ with $\hat{h}(x) \in F[x]$ monic.
If $f$ is irreducible in $R$, then $\hat{h}(x)$ is irreducible in $F[x]$.
If $\hat{h}(x)$ is irreducible in $F[x]$, then $h$ generates a maximal two-sided ideal $Rh$ in $R$ \cite[p.~16]{J96}.

 Two nonzero skew polynomials $f$ and $g$ are  \emph{similar}, written $f\sim g$, if $R/Rf\cong R/Rg$ as left $R$-modules, equivalently, there exist $u,v\in R$, such that $gcrd(f,u)=1$, $gcld(g,v)=1$ and $gu=vf$  \cite[p.~11]{J96}.  The element $u\in R$ can be chosen such that $u\in R_m$ \cite[p.~15]{J96}.

\begin{lemma}\label{le:Similar implies same mclm} (\cite[p.~9, Corollary 2]{Carcanague} and \cite[Theorem 1.2.9]{J96})
Let $f\in R$ have minimal central left multiple $h(t)=\hat{h}(t^n)$.
\\ (i) If $f$ is irreducible  then every $g\in R$ similar to $f$ has $h$ as its minimal central left multiple.
\\ (ii) Suppose that $\hat{h}\in F[x]$ is irreducible. Then $f=f_1\cdots f_r$ for irreducible $f_i\in R$ such that $f_i\sim f_j$ for all $i,j$.
\end{lemma}

The quotient algebra $R/Rh$ has the commutative $F$-algebra $C(R/Rh) \cong F[x]/ (\hat{h}(x))$ of dimension $deg(\hat{h})$ over $F$ as its center, cf. \cite[Lemma 4.2]{GLN13}.
Define $E_{\hat{h}}=F[x]/(\hat{h}(x))$. If $\hat{h}$ is irreducible in $F[x]$, then $E_{\hat{h}}$ is a field extension of $F$ of degree ${\rm deg}(\hat{h})$.

 \cite[Lemma 3]{LS} generalizes as follows for any base field $F$:

\begin{theorem}\label{thm:main2}  \cite{TP21}
Suppose that $f$ is irreducible with minimal central left multiple $h$. Let $k$ be the number of irreducible factors of $h$ in $R$.
 \\ (i) The right nucleus ${\rm Nuc}_r(R/Rf)$ is a central division algebra over $E_{\hat{h}}$ of degree $s=n/k$, and
 $$ R/Rh \cong M_k({\rm Nuc}_r(R/Rf)).$$
 In particular, this means that ${\rm deg}(\hat{h})=\frac{m}{s}$, ${\rm deg}(h)=\frac{nm}{s}$, and
 $[{\rm Nuc}_r(R/Rf) :F]=ms.$
 \\
 (ii) If $n$ is prime and $f$ not right invariant, then ${\rm Nuc}_r(R/Rf)\cong E_{\hat{h}}.$
In particular, then $[{\rm Nuc}_r(R/Rf) :F]=m$, ${\rm deg}(\hat{h})=m$, and ${\rm deg}(h)=mn$ is maximal.
\end{theorem}

Moreover, if ${\rm deg}(h)=mn$ is maximal (which is always the case when $F$ is a finite field but does not hold in general for any base field) and  $\hat{h}$ is irreducible in $F[x]$, then $f$ is irreducible  \cite[Proposition 4.1]{GLN13}.

 Let $f\in R$ be a monic irreducible polynomial of degree $m>1$ and let $h(t)=\hat{h}(t^n)$ be its minimal central left multiple. Set $B={\rm Nuc}_r(R/Rf)$,
 then ${\rm deg}(h)=km$ and $R/Rh\cong M_{k}(B)$ as $E_{\hat{h}}$-algebras by Theorem \ref{thm:main2}. Let
 $\Psi:R/Rh\to M_k(B),$ $\Psi(a+Rh)=M_a$ be this $E_{\hat{h}}$-algebra isomorphism. For each $M_a\in M_k(B)$, we have the endomorphism
 $L_{M_a}:M_k(B)\to M_k(B)$, $L_{M_a}:X\mapsto M_aX.$

\begin{theorem}  \label{thm: rank of a polynomial km} (\cite[Theorem 6]{TP23})
Suppose that ${\rm deg}(h)=km$, then for all nonzero $a+Rh\in R/Rh$ we have
 $$  {\rm colrank}(M_a)=k-\frac{1}{m}{\rm deg}({\rm gcrd}(a,h)).$$
 In particular, if ${\rm deg}(h)=mn$  then
 $M_a\in M_n(E_{\hat{h}})$
 and
$${\rm rank}(M_a)= n - \frac{1}{m}{\rm deg}({\rm gcrd}(a,h)).$$
 \end{theorem}

For finite fields and thus ${\rm deg}(h)=nm$ maximal, this is \cite[Proposition 7]{LS}.

 We now use Theorem \ref{thm: rank of a polynomial km} to  generalize \cite[Theorem 10]{LS} which was proved for finite fields $K$:

\begin{theorem}\label{thm:irred}
\label{thm:new thm 10}
Let $f,g\in R$ be irreducible of degree $m$.
Then $f$ and $g$ are similar if and only if their bounds $f^*$ and $g^*$ are the same.
 \end{theorem}

 \begin{proof}
 We know that for irreducible $f$ and $g$ the minimal central left multiples are bounds of $f$ and $g$.
  Because of Lemma \ref{le:Similar implies same mclm}, it thus remains to show that if $f$ and $g$ have the same minimal central left multiple   then they are similar.

 Let $f\in R$ be monic and irreducible of degree $m$, and let $h(t)=\hat{h}(t^n)$ be the minimal central left multiple of $f$. We can write $h=af$ for a suitable $a\in R$. Then $R/Rh\cong M_{k}(B)$ as $E_{\hat{h}}$-algebras by Theorem \ref{thm:main2}, where $B={\rm Nuc}_r(S_f)$ is a central simple division algebra over $E_{\hat{h}}$ and
$E_{\hat{h}}$ is a field extension of $F$ of degree ${\rm deg}(\hat{h})$. This yields
the $E_{\hat{h}}$-algebra isomorphism $\Psi:R/Rh\to M_k(B),$ $\Psi(a+Rh)=M_a$. In particular, this implies
$$  {\rm colrank}(M_f)= {\rm colrank}(M_g)=k-\frac{1}{m}{\rm deg}({\rm gcrd}(f,h))=k-\frac{1}{m}{\rm deg}({\rm gcrd}(g,h))=k-1$$
 by Theorem \ref{thm: rank of a polynomial km},
 since $f$ and $g$ both have minimal central left multiple $h$.
 The rest of the proof is now  analogous to the one given for \cite[Theorem 10]{LS}, since $B$ is a division algebra (and for this we need the assumption that $f$ and $g$ are irreducible).
 The equality of the column ranks shows that there exist invertible
 matrices $M_u, M_v\in M_k(B)$
 such that $M_u M_f=M_gM_v$, that means $uf=gv\, {\rm mod} \, h$. This implies that there is $b\in R$, such that
 $$gv=uf+bh=uf+baf=(u+ba)f.$$
  Write $v=v'+cf$ for some nonzero $v'\in R$ of degree less than $m$. Then we obtain
 $g(v'+cf)=(u+ba)f$ hence $gv'=(u+ba-gc)f$ and we can conclude that $gv'=u'f$
 for $u'=u+ba-gc$. We have $v'\neq 0$, since $v'=0$ implies $M_v=M_cM_f$ has column rank $\leq k-1$, contradicting the invertibility of $M_v$. We have
$\gcd_r(f,v')=1$, since $f$ is irreducible and $0\leq\deg(v')<m$, and e have
$\gcd_l(g,u')=1$, since assuming $u'=gw$  gives $v'=wf$, which is impossible when comparing the degrees.
This means $g\sim f$.
\end{proof}

 Thus the  different similarity classes depend on the different irreducible polynomials $\hat{h}\in F[x]$ of degree $m_0=mk/n$
 (with $x=t^n$), where $k$ is the number of  irreducible factors of $\hat{h}_1(t^n)\in R$.

  The above result also generalizes \cite[Proposition 2.1.17]{CaB} which was proved for finite fields $K$, that means in the case where the degree of $f^*$ is maximal (i.e., $deg(f^*)=mn$): if $f, g\in R$ are irreducible, then the two monic polynomials $f$ and $g$ are similar if and only if $N(f)$ equals $N(g)$ up to an invertible constant, where $N$ is the reduced norm of the cyclic algebra  $(K(x),\widetilde{\sigma},x)$ over $F(x)$  \cite[Proposition 1.4.6]{J96}. Indeed, when $deg(f^*)=mn$ then $B={\rm Nuc}_r(R/Rf)$ is simply a field extension of $F$ and the proof of \cite[Theorem 10]{LS} holds verbatim; in that case we also have that $N(f)$  is the bound of $f$.
 .

We now generalize \cite[Theorem 8]{LS} to arbitrary fields $K$.

\begin{theorem} \label{thm:easy}
Let $f,g\in R$ be two similar monic polynomials.
 If $f$ and $g$ are irreducible then the nonassociative division rings  $R/Rf$ and $R/Rg$ are isotopic.
\end{theorem}

\begin{proof}
 Since $f$ and $g$ are similar, there is $u\in R_m$ such that $gu=0 \,{\rm mod}_r\,f$. Define
 $H:R/Rg \rightarrow R/Rf$, $H(w)=w\circ_f u$. Then  $H$ is an additive map and as in the proof of \cite[Theorem 8]{LS}, $H(w\circ_g z)=w \circ_f H(z)$ for all $w,z\in R/Rg$. For all $a\in F$ we have $H(a \circ_g w)=(a\circ_f w)\circ_f u=a\circ_f (w\circ_f  u)=a\circ_f  H(w)$
 since $R/Rf$ has center $F$, hence $H$ is an $F$-linear map.
 Since $R/Rf$ is a division ring of finite dimension over $F$, $w\circ_f u=w'\circ_f u$ implies $w=w'$ for all $w,w'\in R_m$, so that $H$ is injective, hence surjective as $f$ and $g$ have the same degree (apply the Rank-Nullity Theorem to the $mn$-dimensional $F$-vector space $R_m$).
 Thus $H(w\circ_g  z)=w H(z)$ yields the desired isotopy.
\end{proof}

Define
$$
{\rm Aut}(K)_\sigma=\{\tau\in {\rm Aut}(K) \mid \tau \sigma = \sigma \tau\}.
$$

\begin{lemma}
(i) The $F$-stable automorphisms form the normaliser subgroup
$N_{\Aut(K)}(\langle\sigma\rangle)$, and
$\Aut(K)_\sigma=\Cent_{\Aut(K)}(\sigma)$ is a subgroup of $N_{\Aut(K)}(\langle\sigma\rangle)$, with equality when
$\Aut(K)$ is abelian.
\\ (ii)
Suppose that ${\rm Aut}(K)$ is abelian.
Then the elements of the group ${\rm Aut}(K)_\sigma$ are precisely the automorphisms of $K$ that preserve the subfield $F$ (that is, that restrict to automorphisms of $F$).
\end{lemma}

\begin{proof}
 (i)  Since $\Aut(K)_\sigma$ is the centralizer subgroup $\Cent_{\Aut(K)}(\sigma)$, it is a group.  As $\sigma$ generates $\Gal(K/F)$, it follows that for any $a\in K$, we have $a\in F$ if and only if $\sigma(a)=a$.  Now suppose $\tau\in \Aut(K)_\sigma$.  Then for any $a\in K$, we have
$$
\tau(a) = \tau(\sigma(a))=\sigma(\tau(a)),
$$
and therefore $a\in F$ if and only if $\tau(a)\in F$, so $F$ is preserved by $\tau$. Conversely, if
$\tau$ preserves $F$ then $\tau\sigma\tau^{-1}\in\Gal(K/F)=\langle\sigma\rangle$,
so $\tau\sigma\tau^{-1}=\sigma^{j}$ for some positive integer $j$.
\\ $(ii)$ follows from $(i)$.
\end{proof}

For all $\tau\in {\rm Aut}(K)_\sigma$, $\alpha\in K^\times$, and for any monic $f\in R$ define the monic polynomial
$$f^{(\alpha,\tau)}(t)=N^\sigma_m(\alpha^{-1}) f^\tau(\alpha t) =N^\sigma_m(\alpha^{-1})\sum_{i=0}^m \tau(a_i) N^\sigma_i(\alpha) t^i$$
where for any $i\in \mathbb{N}$, $\tau\in \Aut(K)$ and $\alpha\in K^\times$, we define
$$
N_i^{\tau}(\alpha) = \prod_{j=0}^{i-1}\tau^j(\alpha).
$$
For $g(x)\in F[x]$, $g(x)=\sum_{i=0}^s c_ix^i$ define $g^\tau(x)=\sum_{i=0}^s \tau(c_i) x^i$.

Define the group $G$ as the semidirect product
$$G=N_{K/F}(K^\times)\rtimes {\rm Aut}(K)_\sigma $$
$$=\{(\lambda,\tau)\,|\,\lambda\in N_{K/F}(K^\times),\tau\in {\rm Aut}(K)_\sigma \},$$
 Note that we have $\tau(N_{K/F}(K^\times))=N_{K/F}(K^\times)$, since $\tau$ and $\sigma$ commute
 which means we have
$$(\lambda_1,\tau_1)(\lambda_2,\tau_2)=(\lambda_1\tau_1(\lambda_2),\tau_1\tau_2).$$
Let $I(F,m_0)$ denote the set of monic irreducible polynomials in $F[x]$ of degree $m_0$.
 We define an action of $G$ on $I(F,m_0)$ via
$$g^{(\lambda,\tau)}(x)=\lambda^{-m_0}g^\tau(\lambda x).$$
We know that $R/Rf\cong R/Rf^{(\alpha,\tau)}$ by Proposition \ref{thm:7}. Put $\lambda=N_{K/F}(\alpha)$  then $(\alpha t)^n=\alpha \sigma(\alpha)\cdots \sigma^{n-1}(\alpha) t^n=\lambda t^n$.
If $h(t)=\hat{h}(t^n)$ is the minimal central left multiple of $f$, then
$$\lambda^{-m_0}\hat{h}^\tau(\lambda x)= \hat{h}^{(\lambda,\tau)}( x)$$
   defines the minimal central left multiple of
 $f^\tau(\alpha t)$ by Lemma \ref{le:low}:

\begin{lemma}\label{le:low}
Let $f\in R$ be a monic irreducible polynomial with  minimal central left multiple $h(t)=\hat{h}(t^n)$.
Then for all $\alpha\in K^\times$ and $\tau\in{\rm Aut}(K)_\sigma$,
$$R/Rf\cong R/Rf^{(\alpha,\tau)}$$
and
$$mclm(f^{(\alpha,\tau)})=\hat h^{(\lambda,\tau)}(x)$$
with $\lambda=N_{K/F}(\alpha)$
is the minimal central left multiple  of $f^{(\alpha,\tau)}$.
\end{lemma}

\begin{proof}
For all $\alpha\in K^\times$ and $\tau\in{\rm Aut}(K)_\sigma$,  we have $R/Rf\cong R/R f^\tau(\alpha t)$ via $G_{\tau,\alpha}$ by Proposition \ref{thm:7}, hence also $R/Rf\cong R/R f^{(\alpha,\tau)}$.

Let us compute the bound, i.e. the minimal central left multiple, of $f^{(\alpha,\tau)}$, note it is the same as the one of the not necessarily monic polynomial $f^\tau(\alpha t)$.

Define $\hat{h_0}(t^n)= mclm(f^\tau(\alpha t))$.
 Since $G_{\tau,\alpha}(f)=f^\tau(\alpha t)$ and $\hat{h}(t^n)=uf$ for some suitable $u\in R$, it follows that
 $$G_{\tau,\alpha}(\hat{h}(t^n))=G_{\tau,\alpha}(u)G_{\tau,\alpha}(f)=G_{\tau,\alpha}(u) f^\tau(\alpha t).$$
 Now compute
  $$G_{\tau,\alpha}(\hat{h}(t^n))=\hat{h}^\tau((\alpha t)^n)=\hat{h}^\tau(\lambda t^n)\in C(R).$$
   This is divisible by $f^\tau(\alpha t)$, so  $\hat{h_0}(x)$ must divide $\hat{h}^\tau(\lambda x)$.
   Suppose that $\hat{h}(t^n)$ decomposes in $R$ into $k$ irreducible factors. Put $m_0=mk/n$.
   Since $G_{\tau,\alpha}:R\to R$ is a ring automorphism with $G_{\tau,\alpha}(C(R))=C(R)$ it maps $Rf\cap C(R)$ bijectively onto $RG_{\tau,\alpha}(f)\cap C(R)$ and preserves degrees. Hence $\deg\hat h'=\deg\hat h$, and the degree of both these monic polynomials is $m_0=mk/n$.
    We obtain that
  $$\hat{h_0}(x)=\lambda^{-m_0} \hat{h}^\tau(\lambda x)$$
  which yields the assertion.
\end{proof}

We now  generalise \cite[Theorem 11]{LS}, but note that our definition of $G$ differs from the one in \cite{LS}, where $\tau\in {\rm Aut}(F)$.

\begin{theorem}\label{thm:11}
\label{thm:11}
Let $f,g\in R$ be irreducible of degree $m$ with  minimal central left multiples $mclm(f)=\hat{h}(t^n)$ and $mclm(g)=\hat{h}_2(t^n)$ with
  $\hat{h},\hat{h}_2\in F[x]$. If $\hat{h}_2 (x)=\hat{h}^{(\lambda,\tau)} (x)$ for some $(\lambda,\tau)\in G$ then the division rings
$R/Rf$ and $R/Rg$ are isotopic.
 \end{theorem}

 \begin{proof}
 Suppose that $\hat{h}(t^n)$ decomposes in $R$ into $k$ irreducible factors. Put $m_0=mk/n$.
  Let $(\lambda,\tau)\in G$ and let $\alpha\in K^\times$ such that $N_{K/F}(\alpha)=\lambda$.

  We assume that  $\hat{h}_2 (x)=\hat{h}^{(\lambda,\tau)} (x)$ for some $(\lambda,\tau)\in G$, that means that $\hat{h}_2 (x)$ lies in the orbit of $\hat{h}(x)$. Since $\hat{h}^{(\lambda,\tau)} (t^n)$ is the bound of $f^{(\alpha,\tau)}$ (Lemma \ref{le:low}) it follows that $f^{(\alpha,\tau)}$ is similar to $g$ by Theorem \ref{thm:irred} and therefore $R/Rf^{(\alpha,\tau)}$ is isotopic to $R/Rg$ as $F$-algebras (Theorem \ref{thm:easy}).

   By  Proposition \ref{thm:7},  $R/Rf\cong R/R{f^\tau(\alpha t)}$ as nonassociative rings, hence also $R/Rf\cong R/Rf^{(\alpha,\tau)}$, and so $R/Rf$ is similar to $ R/Rg$.
   \end{proof}

In particular, note that if $F_0={\rm Fix}(\tau)\cap F$ in the above proof, then  $R/Rf$ and $R/Rf^{(\alpha,\tau)}$ are also isomorphic as $F_0$-algebras, and $R/Rf^{(\alpha,\tau)}$ and $R/Rg$ are isotopic as $F$-algebras by Theorem \ref{thm:easy}, thus $R/Rf$ and $R/Rg$ are isotopic as $F_0$-algebras.

\begin{lemma}\label{le:trivial}
Let  $\hat h(x)=\sum_{i=0}^{m_0} c_i x^i, \hat s(x)=\sum_{i=0}^{m_0} d_i x^i \in I(m_0)$.
Then
\\ (i) $\hat h(x)$ lies in the orbit of $\hat s(x)$ if and only if there exists $(\lambda,\tau)\in G$, such that
$$ c_i=\lambda^{-m_0+i}\tau(d_i)$$
for all $i\in \{ 0,\dots , m_0-1\}$.
\\ (ii) If $\hat h(x)$ lies in the orbit of $\hat s(x)$ then for all $i\in \{ 0,\dots , m_0-1\}$, $c_i=0$ if and only if $d_i=0$.
\end{lemma}

\begin{proof}
$(i)$ The monic polynomial $\hat h(x)$ lies in the orbit of $\hat s(x)$
if and only if there exists $(\lambda,\tau)\in G$ with
$\hat h(x)=\hat s^{(\lambda,\tau)}(x)$. This is the same as saying that there exists $(\lambda,\tau)\in G$ such that
$ c_i=\lambda^{-m_0+i}\tau(d_i)$
for all $i\in \{ 0,\dots , m_0-1\}$.
\\ $(ii)$ follows from $(i)$.
\end{proof}

 \begin{remark} In \cite{LS}, $F=\mathbb{F}_{q}$ and $K=\mathbb{F}_{q^n}$ are finite fields. Let $F_0=\mathbb{F}_p$ be the prime field. We know that ${\rm Gal}(F/F_0)={\rm Gal}(K/F_0)/{\rm Gal}(K/F)$, so that for every $\rho\in {\rm Gal}(F/F_0)$ there are $\tau_i\in  {\rm Gal}(K/F_0)$ with ${\tau_i}|_{F}=\rho$.
Define the group $G'$ as the semidirect product $N_{K/F}(K^\times)\rtimes {\rm Aut}(F)$ as in \cite{LS}, i.e.
$$G'=\{(\lambda,\tau)\,|\,\lambda\in N_{K/F}(K^\times),\tau\in {\rm Aut}(F)  \}.$$
Define an action of $G'$ on $I(F,m_0)$ via
$$g^{(\lambda,\tau)}(x)=\lambda^{-m_0}g^\tau(\lambda x).$$
If $\rho \in {\rm Aut}(F)$, as  assumed in \cite{LS}, for instance $h$ and $\varphi$ in the proof of \cite[Theorem 11]{LS} would not be defined. We assume the authors mean they choose a fixed extension $\tau$ of $\rho $ to ${\rm Aut}(K)$.
\end{remark}

\begin{theorem}
\label{thm:main12}
Let $f,g\in R$ be irreducible of degree $m$ with  minimal central left multiples $mclm(f)=\hat{h}_1(t^n)$ and $mclm(g)=\hat{h}_2(t^n)$
  with $\hat{h}_1,\hat{h}_2\in F[x]$. If $\hat{h}_1$ and $\hat{h}_2$ lie in the same $G$-orbit of $I(F,m_0)$, then $R/Rf$ and $R/Rg$  are isotopic division rings. Moreover, for all $\tau_i\in {\rm Aut}(K)_\sigma$ such that ${\tau_i}|_F={\tau}|_F$, and all $\alpha\in K^\times$ with $N_{K/F}(\alpha)=\lambda$,  
  $R/Rf^{(\alpha,\tau_i)}$ is isotopic to $R/Rf$ and $R/Rg$
\end{theorem}

  This generalises \cite[Theorem 12]{LS} and reformulates Theorem \ref{thm:11}. The  proof given in \cite{LS} can be adjusted to hold for any base field and our group $G$.

 \begin{proof}
Let $f,g\in R$ be irreducible of degree $m$ with  minimal central left multiples $mclm(f)=\hat{h}_1(t^n)$ and $mclm(g)=\hat{h}_2(t^n)$ with $\hat{h}_1,\hat{h}_2\in F[x]$  irreducible of degree $m_0$. If $\hat{h}_1(x)$ and $\hat{h}_2(x)$  lie in the same orbit of $G$, defined by $(\lambda,\tau)$,  then $R/Rf$, $R/Rg$  and $R/R f^{(\alpha, \tau_i)}$  are isotopic division rings for all $\tau_i\in {\rm Aut}(K)_\sigma$ such that ${\tau_i}|_F={\tau}|_F$, and all $\alpha\in K^\times$ with $N_{K/F}(\alpha)=\lambda$ (Theorem \ref{thm:11}).  The bound of $f^{(\alpha,\tau_i)}$ is $\hat h_1^{(\lambda,\tau)}=\hat h_2$, so that $R/Rf^{(\alpha,\tau_i)}$ is isotopic to $R/Rf$ and $R/Rg$. Note that $\hat h^{(\lambda,\tau_i)}$ depends only on $\tau_i|_F$, since $\hat h\in F[x]$.
\end{proof}

When $f(t)=t^m-a\in K[t;\sigma]$  with  minimal central left multiples $mclm(f)=\hat{h}(t^n)$  and $n\geq m$, then  all the algebras whose bounds lie in the same orbit as $\hat h(x)$  will not just be isotopic, but indeed isomorphic.
 In the literature, the opposite algebras of the Petit algebras $K[t;\sigma]/K[t;\sigma](t^m-c)$ are also called \emph{Sandler algebras} \cite{Sandler1962}.

\begin{theorem}\label{thm:maincor}
Suppose that  ${\rm Aut}(K)$ is abelian. Let
$f(t)=t^m-c, g(t)=t^m-d \in K[t;\sigma]$ be irreducible with  minimal central left multiples $mclm(f)=\hat{h}(t^n)$ and $mclm(g)=\hat{s}(t^n)$,
  $\hat{h},\hat{s}\in F[x]$. Suppose that
  either  $n > m\geq 3$ and $\sigma(c)\neq c$ and $\sigma(d)\neq d$, 
  or that
  $n=m$ and $c$ and $d$  do not lie in any proper subfield of $K$. If $\hat{s} (x)=\hat{h}^{(\lambda,\tau)} (x)$ for some $(\lambda,\tau)\in G$ then
$K[t;\sigma]/K[t;\sigma](t^m-c)\cong K[t;\sigma]/K[t;\sigma](t^m-d)$ are properly nonassociative isomorphic algebras.
\end{theorem}

\begin{proof}
We know that $K[t;\sigma]/K[t;\sigma](t^m-c)$ and
 $K[t;\sigma]/K[t;\sigma](t^m-d)$ are isotopic by Theorem \ref{thm:11} and proper nonassociative algebras since $n > m$ and  $c\in K\setminus F$, or   $n=m$ and $c$  does not lie in any proper subfield of $K$.
By assumption,
$K[t;\sigma]/K[t;\sigma](t^m-c)$ and
 $K[t;\sigma]/K[t;\sigma](t^m-d)$ are two proper nonassociative Petit division algebras.  If $K[t;\sigma]/K[t;\sigma](t^m-c)$ and
 $K[t;\sigma]/K[t;\sigma](t^m-d)$ are isotopic, then
$K[t;\sigma]/K[t;\sigma](t^m-c)\cong K[t;\sigma]/K[t;\sigma](t^m-d)$ by \cite[Section 4]{NevinsPumpluen2025San}.
\end{proof}

Fix some $K/F$.
 If we want to construct a similarity class
 of Petit division algebras, we first choose a  monic irreducible polynomial in $F[x]$ of degree $m_0$, this will be
 $\hat{h}(x)\in I(F,m_0)$ and represent one $G$-orbit of $I(F,m_0)$. Put $h(t)=\hat{h}(t^n)$ and compute the number $k$ of irreducible factors of $h(t)$ in $R$. Since  $m_0=mk/n$ this means $m=m_0n/k$ and so the irreducible skew polynomials  with bound $h(t)=\hat{h}(t^n)$ have degree $m_0n/k$. We then calculate an irreducible divisor $f\in R$ of  $\hat{h}(t^n)\in R$. The division ring $R/Rf$ then represents the isotopy class that is represented by the $G$-orbit $\hat{h}\in I(F,m_0)$ is contained in. Unlike in the finite field case, here the degree of $f$ and hence the dimension of the Petit algebra $R/Rf$ over $F$, that represents the isotopy class that is represented by the $G$-orbit $\hat{h}(x)\in I(F,m_0)$, can vary depending on the chosen element $\hat{h}(x)$.

 We now give some examples  where the degree of $f$ is the same as the degree of $\hat{h}\in I(F,m_0)$, i.e. $m=m_0$.

\begin{proposition} \label{prop:main2}
Let $f$ be monic  and irreducible of degree $m$.
\\ (i)  Suppose that $f$ and not right invariant.
 If ${\rm gcd}(m,n)=1$, then the minimal central left multiple of $f$ has maximal degree $mn$.
 \\ (ii) Suppose that $m$ is prime. Assume that the minimal central left multiple satisfies that $\hat{h}(x)\not=x+a\in F[x]$   for any $a\in F$.
 Then $f$ is not right invariant and the minimal central left multiple $h$ of $f$ has maximal degree $mn$.
 \\ In these cases, $m_0=m$.
\end{proposition}

\begin{proof}
$(i)$
Since ${\rm gcd}(m,n)=1$,  ${\rm Nuc}_r(R/Rf)\cong E_{\hat{h}}$ is a field extension of $F$ of degree $m$  \cite[Corollary 29]{Ow}. Thus we know that ${\rm deg}(h)=mn$.
\\ $(ii)$  By \cite[Corollary 30]{Ow},
 $f$ is not right invariant and one of the following holds: either ${\rm Nuc}_r(S_f)\cong E_{\hat{h}}$,
 $[{\rm Nuc}_r(R/Rf) :F]=m$, and ${\rm deg}(h)=mn$, or we have
 the second case, that ${\rm Nuc}_r(R/Rf)$ is a central division algebra over $F=E_{\hat{h}}$ of prime degree $m$,
 and $m$ divides $n$. This second case occurs when $\hat{h}(x)=x+a\in F[x]$. This proves our assertion.
\end{proof}

\subsection{The associated MRD codes}

There is a one-one correspondence between finite-dimensional nonassociative division algebras  and linear square maximum-rank distance (MRD) codes  \cite{CKWW2015}.
We can consider every algebra $A$  as a  left module over its left nucleus $N$.
   The right multiplication by non-zero elements in a finite-dimensional division algebra $A$ over a field $F$ then yields  an
    associated rank metric code $\{R_a\,|\,  a\in A\setminus \{0\}\}\subset {\rm End}_N(A)$  of bijective $F$-linear maps. If $\dim_N A=k$ then this spread set canonically yields a rank-metric code $C(A)\subset M_k(N)$ which is $F$-linear and has minimum distance $k$, since every matrix in it has column rank $k$. This set of matrices is a MRD code over $N$. We know that for our Petit algebras we have $N=K$ and $k=m$.

   Two linear codes $C$ and $C'$ in $M_m( K)$, $K$ a field, are called \emph{equivalent}, if there exists some $\varphi\in {\rm Aut}(M_m( K))$
and invertible $X,Y\in M_m(K),$ such that $C'= X C^\varphi Y$.
Equivalent matrix codes have the same behaviour and invariants, so in order not to duplicate known matrix codes it is important to know their equivalence classes. If $R/Rf$ and $R/Rg$ are isotopic, then the associated MRD matrix codes $C=\{R_a\,|\,  a\in R/Rf\}\subset M_m(K)$ and $C'=\{R_a\,|\,  a\in R/Rg\}\subset M_m(K)$ are equivalent (Theorem \ref{t:main2}). These matrix codes are examples of the codes proposed in
\cite{Sheekey2019}

Let $f$ and $g$ be monic and irreducible of degree $m$ in $K[t;\sigma]$.
Let $f^*=\hat{h}_1(t^n)$ and $g^*=\hat{h}_2(t^n)$, $\hat{h}_1,\hat{h}_2\in F[x]$, be the  minimal central left multiples of $f$ and $g$, hence the bounds. Then $f\sim g$ if and only if there exists an invertible matrix $B_0\in  M_m(K)$, such that
$A_f B_0=B_0^\sigma A_g$, where $A_f$ denotes the companion matrix of $f$,
if and only if $f^*=g^*$. Here, $C^\sigma$ denotes the matrix $C$ where $\sigma$ is applied to each entry. Moreover, then $R/Rf$ and $R/Rg$ are isotopic algebras  (Theorem \ref{thm:easy}).
    If $R/Rf$ and $R/Rg$ are isotopic algebras then the associated rank metric codes $C=\{R_a\,|\,  a\in R/Rf\}\subset M_m(K)$ and $C'=\{R_a\,|\,  a\in R/Rg\}\subset M_m(K)$ are equivalent.

  We proved that if $\hat{h}_2 (x)=\hat{h}_1^{(\lambda,\tau)} (x)$ for some $(\lambda,\tau)\in G$ then  $R/Rf$ and $R/Rg$ are
  isotopic.

  \begin{theorem}\label{t:main2}
 \textcolor[rgb]{1.00,0.00,0.00}{Let $f,g$ be monic and irreducible of degree $m$.} If $\hat{h}_1$ and $\hat{h}_2$ lie in the same $G$-orbit of $I(F,m_0)$, then for $f$ and $g$ with $f^*=\hat{h}_1(t^n)$ and $g^*=\hat{h}_2(t^n)$,
  the associated MRD codes $C=\{R_a\,|\,  a\in S_f\}\subset M_k(K)$ and $C'=\{R_a\,|\,  a\in S_g\}\subset M_m(K)$ are equivalent.
\end{theorem}

\begin{proof}
 The isotopy of Theorem \ref{thm:easy} is $({\rm id},H,H)$ with $H(w)=w\circ_f u$. Since $K={\rm Nuc}_l(R/Rf)$, $H$ is left $K$-linear. From $H(w\circ_g z)=w\circ_f H(z)$ we get $R^{(f)}_{H(z)}=H\circ R^{(g)}_z$, i.e. $C_f=HC_g$. Note that the isomorphism $G_{\tau,\alpha}$ is $\tau$-semilinear, so it yields $C^\tau$ which is equivalent to $C$.
\end{proof}

\section{Finite fields} \label{sec:finite}

Let  $K=\mathbb{F}_{q^n}$, $M=\mathbb{F}_{q^m}$ with ${\rm Gal}(M/F)=\langle \varphi \rangle$ $\varphi(c)=c^{q}$, $F=\mathbb{F}_{q}=\mathbb{F}_{p^h}$ and  $F_0=\mathbb{F}_{p}$. Then
$$G=\mathbb{F}_{q}^\times\rtimes {\rm Gal}(\mathbb{F}_{q^n}/\mathbb{F}_{p}).$$
Let $A(q,n,m)$ denote the number of isotopy classes of semifields $R/Rf$ of order $q^{mn}$ defined by monic irreducible $f\in \mathbb{F}_{q^n}[t;\sigma]$ of degree $m$ and
 $N(q,m)=|I(q,m)|$ denote the number of irreducible monic polynomials of degree $m$ in $\mathbb{F}_q[x]$. Then
 $$N(q,m)=\frac{q^m-\theta}{m}=\frac{1}{m}\sum_{s|m}\mu(s)q^{m/s},$$
 where $\theta$ is the number of elements in $\mathbb{F}_{q^m}$ contained in a proper subfield and $\mu$ is the Moebius function.

Let $M(q,m)$ denote the number of orbits in $I(q,m)$ under the action of $G$.
We reformulate \cite[Theorem 12]{LS}, as counting the number of $G'$-orbits  requires the same proof.

\begin{corollary} \label{thm:12}
 The number of isotopy classes of semifields $R/Rf$ of order $q^{nm}$ obtained from irreducible skew polynomials $f$ of degree $m$ in $ \mathbb{F}_{q^n}[t;\sigma]$ is less than or equal to the number $M(q,m)$ of $G$-orbits  on the set $I(q,m)$ of monic irreducible polynomials in $\mathbb{F}_q[x]$ of degree $m$.
 \end{corollary}

The upper bound
 $A(q,n,m)\leq N(q,m)=\frac{q^m-\theta}{m}$ was strengthened to
 $$A(q,n,m)\leq M(q,m)$$
in \cite{LS} where it was also noted that
$$\frac{q^m-\theta}{mh(q-1)} \leq M(q,m)\leq \frac{q^m-\theta}{m}.$$

Let us now compute $M(q,m)$ explicitly.
Let $\tau\in{\rm Gal}(K/F_0)={\rm Gal}(\mathbb{F}_{q^n}/\mathbb{F}_p)$.
Since $\mathbb{F}_{q^n}=\mathbb{F}_{p^{hn}}$, we have $\tau=\tau_r$ with $\tau_r(u)=u^{p^r}$ for a unique
integer $r$, $0\leq r<hn$.

The action of $(\lambda,\tau_r)\in G$ on $I(q,m)$ only involves the restriction of $\tau_r$ to $F=\mathbb{F}_q$,
since $h^{(\lambda,\tau_r)}(x)=\lambda^{-m}h^{\tau_r}(\lambda x)$ and the coefficients of $h$ lie in $\mathbb{F}_q$.
As $\tau_r|_{\mathbb{F}_q}(c)=c^{p^r}$ and $c^{p^h}=c$ for all $c\in\mathbb{F}_q$, this restriction, and hence
the action of $(\lambda,\tau_r)$ on $I(q,m)$, depends only on $r$ modulo $h$.

To describe the fixed points we work with the roots of $h\in I(m)$, which lie in $M=\mathbb{F}_{q^m}$.
The automorphism $\mathbb{F}_{q^m}\to \mathbb{F}_{q^m}$, $a\mapsto a^{p^r}$, extends $\tau_r|_{\mathbb{F}_q}$ to $\mathbb{F}_{q^m}$. This
extension depends on $r$ modulo $hm$. If $r\equiv r' \,{\rm mod }\, h$, then $a^{p^{r'}}=(a^{p^r})^{q^i}$
for some $i$, so the two extensions differ by a power of the Frobenius automorphism $a\mapsto a^q$. Since the
Frobenius automorphism permutes the roots of each $h\in I(q,m)$, the choice of extension does not affect the
fixed point condition below.

We will look at the elements in $I(q,m)$ in our argument, i.e. at monic irreducible
 $h(x)=\sum_{i=0}^{m-1}h_ix^i\in \mathbb{F}_{q}[x]$ of degree $m$,  and at their roots $a\in\mathbb{F}_{q^m}$.

 If $h(a)=0$ then $h$ is the minimal polynomial of $a$, hence $\mathbb{F}_{q}(a)\cong \mathbb{F}_{q^m}$, and $\{a,a^{q^2},\dots,a^{q^{m-1}}\}$ are the $m$ distinct roots of $h$. There is a one-to-one correspondence between the monic irreducible  polynomials $h(x)\in I(q,m)$ and
  the sets  $\{a,a^q,\dots,a^{q^{m-1}}\}$, where $a\in \mathbb F_{q^m}$
  with $\mathbb F_q(a)=\mathbb{F}_{q^{m}}$. 
 
 If $h(c)=0$ then $h^{(\lambda,\tau)}(\lambda^{-1}c^{p^r})=0$, too, so that $h^{(\lambda,\tau)}=h$ if and only if the set $\{a,a^{q^2},\dots,a^{q^{m-1}}\}$ is invariant under the operation $a\mapsto \lambda^{-1}\tau(a)=\lambda^{-1}a^{p^r}$. This means there exists some integer $j$, $0\leq j\leq m-1$, such that
 $$\lambda^{-1}\tau(a)=\lambda^{-1}a^{p^r}=a^{q^j}$$
  which is equivalent to $\lambda= a^{p^r-q^j}$. Put $e_{r,j}=p^r-q^j$
  (note that $\gcd(q^s-1,e_{r,j})$ means $\gcd(q^s-1,|e_{r,j}|)$, which equals $q^s-1$ if $e_{r,j}=0$).

 For every positive integer $s$ that divides $m$,
  $\lambda= a^{e_{r,j}}$ has either 0 solutions or exactly $gcd(q^s-1,e_{r,j})$ solutions in $\mathbb{F}_{q^s}$. Let  $N_{r,j}(s,\lambda)$ be the number of solutions of $ \tau(a)=\lambda a^{q^j}$ that lie in the subfield $\mathbb{F}_{q^{s}}$ for any $s$ where $s|m$. Then
  $$N_{r,j}(s,\lambda)=gcd(q^s-1,e_{r,j}) \text{ if }\lambda^{(q^s-1)/gcd(q^s-1,e_{r,j})}=1$$
   and 0 otherwise.

   We  determine the number $S_{r,j}$ of all solutions of $ \tau(a)=a^{p^r}=\lambda a^{q^j}$ for any $a\in \mathbb{F}_{q^m}$, which has a minimal polynomial of degree $m$ and where we consider all possible intermediate fields of $\mathbb{F}_{q^m}/\mathbb{F}_{q}$, for fixed $j$ and $r$. We get
 $$S_{r,j}(\lambda)=\sum_{d|m}\mu(d) N_{(r,j)}(\frac{m}{d},\lambda)$$
 by a straightforward inclusion/exclusion counting argument.

 Define  ${\rm Fix}(\lambda,\tau)=\{ h\in I(q,m)\,|\, h^{ (\lambda,\tau) }=h \}$. Each monic irreducible polynomial corresponds to exactly $m$ roots, so the number of  monic irreducible polynomials that are fixed by the action of $(\lambda,\tau)$ is
 $$|{\rm Fix}(\lambda,\tau)|$$
 $$=\frac{1}{m}|\{ a\in \mathbb{F}_{q^m} \,|\, [\mathbb{F}_{q}(a):\mathbb{F}_{q}]=m \text{ and there is integer } j,  0\leq j\leq m-1  \text{ with } \tau(a)=\lambda a^{q^j} \}| $$
 $$=\frac{1}{m}\sum_{j=0}^{m-1} S_{r,j}(\lambda).$$
 We note that taking the sum over $j$ counts each  $a$ exactly once, since $a,a^q,\dots,a^{q^{m-1}}$ are pairwise distinct.
The number $M(q,m)$  of distinct $G$-orbits on $I(q,m)$ is given by
 $$M(q,m)=\frac{1}{|G|}\sum_{(\lambda,\tau)\in G} |{\rm Fix}(\lambda,\tau)|=\frac{1}{(q-1)hn}\sum_{(\lambda,\tau)\in G} |{\rm Fix}(\lambda,\tau)|.$$
 Write $\tau_r$ for $\tau(u)=u^{p^r}$. We obtain
 $$M(q,m)=\frac{1}{(q-1)hn}
 \sum_{\lambda\in \mathbb{F}_q^\times}\sum_{r=0}^{hn-1} |{\rm Fix}(\lambda,\tau_{r})|$$
 $$=\frac{1}{(q-1)hn}\sum_{\lambda\in \mathbb{F}_q^\times}\sum_{r=0}^{hn-1} \frac{1}{m} \sum_{j=0}^{m-1}
 \sum_{d|m}\mu(d) N_{(r,j)}(\frac{m}{d},\lambda).$$
 Now $|{\rm Fix}(\lambda,\tau)|$ only depends on $\tau|_F\in {\rm Gal}(\mathbb{F}_{q}/\mathbb{F}_{p})$, and for every one of the $h$ different $\widetilde{\tau}\in {\rm Gal}(\mathbb{F}_{q}/\mathbb{F}_{p})$,
 we have exactly $n$ possible extensions of  $\widetilde{\tau}$ to some $\tau\in {\rm Gal}(\mathbb{F}_{q^n}/\mathbb{F}_{p})$. Each extension acts differently on $K$ viewed as automorphism on $K$, but yields the same $|{\rm Fix}(\lambda,\tau)|$.
 We obtain that in fact
 $$\sum_{r=0}^{hn-1} |\mathrm{Fix}(\lambda,\tau_{r})|
= n \sum_{k=0}^{h-1} |\mathrm{Fix}(\lambda,\widetilde{\tau}_{k})|,$$
so that we can also express
$$ M(q,m) = \frac{1}{(q-1) h} \sum_{\lambda \in \mathbb{F}_q^\times} \sum_{k=0}^{h-1} |\mathrm{Fix}(\lambda,\widetilde{\tau}_{k})|$$
$$= \frac{1}{(q-1)hm}\sum_{\lambda\in \mathbb{F}_q^\times}\sum_{k=0}^{h-1}  \sum_{j=0}^{m-1}
 \sum_{d|m}\mu(d) N_{(k,j)}(\frac{m}{d},\lambda) ,$$
 where we are now counting the $N_{(k,j)}(\frac{m}{d},\lambda) $ with respect to the $h$ different choices that are possible for $\widetilde{\tau}$.

 \begin{theorem}\label{thm:countingorbits}
 The number  of distinct $G$-orbits on $I(q,m)$ is
$$M(q,m)=\frac{1}{(q-1)hm}\sum_{\lambda\in \mathbb{F}_q^\times}\sum_{r=0}^{h-1} \sum_{j=0}^{m-1}
 \sum_{d|m}\mu(d) N_{(r,j)}(\frac{m}{d},\lambda).$$
  \end{theorem}

  Santonastaso and Zhou prove that for $n>2$ the number of isotopy classes equals the number of distinct $G$-orbits on $I(q,m)$  \cite[Corollary 5.12]{SantonastasoZhou2026} and proceed to give  a lower bound for that number, i.e. the number of isotopy classes.
Combined with \cite[Corollary 5.12]{SantonastasoZhou2026},  the number of distinct $G$-orbits $M(q,m)$ from Theorem \ref{thm:countingorbits} thus gives us the exact number of isotopy classes for $n>2$.

\begin{theorem} \label{thm:mainmain}
 Let $n>2$ and $m\geq2$, fix $\sigma$. Then the number of isotopy classes of semifields $R/Rf$ of order $q^{nm}$ obtained from irreducible skew polynomials $f$ of degree $m$ in $ \mathbb{F}_{q^n}[t;\sigma]$ is exactly
 $$M(q,m)=\frac{1}{(q-1)hm}\sum_{\lambda\in \mathbb{F}_q^\times}\sum_{r=0}^{h-1}  \sum_{j=0}^{m-1}
 \sum_{d|m}\mu(d) N_{(r,j)}(\frac{m}{d},\lambda).$$
 \end{theorem}

  Note that using \cite{KL},  for all $n>2 $, the total number of isotopy classes of semifields $S_f$ defined by irreducible polynomials $f\in \mathbb{F}_{q^n}[t;\sigma]$ of degree $m$ for all choices of $\sigma$ with ${\rm Fix}(\sigma)=\mathbb{F}_q$ is bounded above by
$$\frac{\varphi(n)}{2}M(q,m)$$
(\cite[Remark 6]{LS}).
Here $\varphi$ is Euler's totient function.

\subsection*{Acknowledgments} The first draft of this paper was written during the  author's stay as a Simons Professor in Residence at the University of Ottawa in July and August 2025. She gratefully acknowledges the support of CRM and the Simons Foundation, and thanks the Department of Mathematics and Statistics for its hospitality and its congenial and inspiring atmosphere. Particular thanks go to Monica Nevins for the many fruitful and inspiring discussions that helped shape this paper.

%%%%%%%%%%%%%%%%%%%%%%%%%%%%%%%%

\providecommand{\bysame}{\leavevmode\hbox to3em{\hrulefill}\thinspace}
\providecommand{\MR}{\relax\ifhmode\unskip\space\fi MR }
\providecommand{\MRhref}[2]{%
  \href{http://www.ams.org/mathscinet-getitem?mr=#1}{#2}
}
\providecommand{\href}[2]{#2}

\end{document}